\documentclass[11pt, twoside]{article}
\pdfoutput=1

\usepackage{graphicx}
\usepackage[caption=false]{subfig}
\usepackage{amsmath}
\usepackage{amssymb,amsfonts}
\usepackage{amsthm}
\usepackage{bm}
\usepackage{mathrsfs}
\usepackage{amssymb}

\usepackage[margin=1in]{geometry}

\usepackage{algorithm}
\usepackage{algorithmic}

\numberwithin{equation}{section}

\theoremstyle{definition}
\newtheorem{theorem}{Theorem}

\newtheorem{proposition}{Proposition}
\newtheorem{lemma}{Lemma}
\newtheorem{corollary}{Corollary}
\newtheorem{definition}{Definition}
\newtheorem{remark}{Remark}

\usepackage{cite}
\usepackage{hyperref}
\usepackage[nameinlink]{cleveref}

\usepackage{fancyhdr}
\begin{document}
	
	\title{Error Analysis of an Incremental POD Algorithm for PDE Simulation Data}

\author{Hiba Fareed%
	\thanks{Department of Mathematics and Statistics, Missouri University of Science and Technology, Rolla, MO (\mbox{hf3n3@mst.edu}, \mbox{singlerj@mst.edu}).}
	\and
	John~R.~Singler%
	\footnotemark[1]
}

\maketitle

\begin{abstract}
	In our earlier work \cite{Fareed17}, we proposed an incremental SVD algorithm with respect to a weighted inner product to compute the proper orthogonal decomposition (POD) of a set of simulation data for a partial differential equation (PDE) without storing the data.  In this work, we perform an error analysis of the incremental SVD algorithm.  We also modify the algorithm to incrementally update both the SVD and an error bound when a new column of data is added.  We show the algorithm produces the exact SVD of an approximate data matrix, and the operator norm error between the approximate and exact data matrices is bounded above by the computed error bound.  This error bound also allows us to bound the error in the incrementally computed singular values and singular vectors.  We illustrate our analysis with numerical results for three simulation data sets from a 1D FitzHugh-Nagumo PDE system with various choices of the algorithm truncation tolerances.
\end{abstract}

\section{Introduction}

Proper orthogonal decomposition (POD) is a method to find an optimal low order basis to approximate a given set of data.  The basis elements are called POD modes, and they are often used to create low order models of high-dimensional systems of ordinary differential equations or partial differential equations (PDEs) that can be simulated easily and even used for real-time applications.  For more about the applications of POD in engineering and applied sciences and POD model order reduction, see, e.g., \cite{colonius02, ZIMMERMANN10, Zimmermann14, Peng16, Christensen99, Kalashnikova14, Amsallem16, Daescu08, Holmes12, Barone09, Calo14, Farhat15, XieWellsWangIliescu18, MohebujjamanRebholzXieIliescu17, GunzburgerJiangSchneier17, Kostova-VassilevskaOxberry18}.

There is a close relationship between the singular value decomposition (SVD) of a set of data and the POD eigenvalues and modes of the data.  Due to applications involving functional data and PDEs, many researchers discuss this relationship in weighted inner product spaces and general Hilbert spaces \cite{QuarteroniManzoniNegri16, Singler14, KunischVolkwein02, GubischVolkwein17}.  For the POD calculation, it is important to determine an inner product that is appropriate for the application \cite{Zlatko17, Tabandeh16, Serre12, Amsallem15, Kalashnikova14}.

Since the size of data sets continues to increase in applications, many researchers have proposed and developed more efficient algorithms for POD computations, the SVD, and other related methods \cite{Brand02, Brand06, BakerGallivan12, Chahlaoui03, Iwen16, Mastronardi05, Mastronardi08, Fahl01, BeattieBorggaard06, WangMcBeeIliescu16, HimpeLeibnerRave_pp}.  These algorithms have been recently applied in conjunction with techniques such as POD model order reduction and the dynamic mode decomposition, which often consider simulation data from a PDE \cite{PlaczekTranOhayon11, Amsallem15, CoriglianoDossiMariani15, PeherstorferWillcox15, PeherstorferWillcox16, Schmidt17, ZahrFarhat15, Zimmermann17, ZimmermannPeherstorferWillcox17, NME:NME5283}.

In our earlier work \cite{Fareed17}, we proposed an incremental SVD algorithm for computing POD eigenvalues and modes in a weighted inner product space.  Specifically, we considered Galerkin-type PDE simulation data, initialized the SVD on a small amount of the data, and then used an incremental approach to approximately update the ٍSVD with respect to a weighted inner product as new data arrives.  The algorithm involves minimal data storage; the PDE simulation data does not need to be stored.  The algorithm also involves truncation, and therefore produces approximate POD eigenvalues and modes.  We proved the SVD update is exact without truncation.

In this paper, we study the effectiveness of the truncations and deduce error bounds for the SVD approximation.  To handle the computational challenge raised by large data sets, we bound the error incrementally.  Specifically, we extend the incremental SVD algorithm for a weighted inner product in \cite{Fareed17} to compute an error bound incrementally without storing the data set; see \Cref{Section:basic}, \Cref{algorithm:incrSVD_Error_weightedinner}.  We also perform an error analysis in \Cref{subsec:Error Analysis} that clarifies the effect of truncation at each step, and provides more insight into the accuracy of the algorithm with truncation and the choices of the two tolerances.  We prove the algorithm produces the exact SVD of an approximate data set, and the operator norm error between the exact and approximate data set is bounded above by the incrementally computed error bound.  This yields error bounds for the approximate POD eigenvalues and modes.  To illustrate the analysis, we present numerical results in \Cref{sec:numerical_results} for a set of PDE simulation data using various choices of the tolerances. Finally, we present conclusions in \Cref{sec:conclusion}.

\section{Background and Algorithm}
\label{Section:basic}

We begin by setting notation, recalling background material, and discussing the algorithm.

For a matrix $A \in \mathbb{R}^{m \times n}$, let $ A_{(p : q, r : s)} $ denote the submatrix of $ A $ consisting of the entries of $ A $ from rows $ p, \ldots, q $ and columns $ r, \ldots, s $.  Also, if $ p $ and $ q $ are omitted, then the submatrix should consist of the entries from all rows.  A similar convention applies for the columns if $ r $ and $ s $ are omitted.

Let $ M \in \mathbb{R}^{m \times m} $ be symmetric positive definite, and let $ \mathbb{R}^m_M $ denote the Hilbert space $ \mathbb{R}^m $ with weighted inner product $ (x,y)_M = y^T M x $ and corresponding norm $ \| x \|_{ M } = ( x^T M x )^{1/2} $.  For a matrix $ P \in \mathbb{R}^{m \times n} $, we can consider $ P $ as a linear operator $ P : \mathbb{R}^{n} \to \mathbb{R}_{M}^{m}$.  In this case, the operator norm of $ P $ is
$$
  \| P \|_{ \mathcal{L}( \mathbb{R}^{n}, \mathbb{R}_{M}^{m}) }= \sup_{\|x\|=1} \| P x \|_{ M }.
$$
We note that $ \mathbb{R}^n $ without a subscript should be understood to have the standard inner product $ (x,y) = y^T x $ and Euclidean norm $ \| x \| = (x^T x)^{1/2} $.  The Hilbert adjoint operator of the matrix $ P: \mathbb{R}^{n} \to \mathbb{R}_{M}^{m}$ is the matrix $P^{*} : \mathbb{R}_{M}^{m} \to \mathbb{R}^{n}$ given by $P^* = P^{T} M $.  We have $ (Px,y)_M = (x,P^*y) $ for all $ x \in \mathbb{R}^n $ and $ y \in \mathbb{R}^m_M $.

In our earlier work \cite{Fareed17}, we discussed how the proper orthogonal decomposition of a set of PDE simulation data can be reformulated as the SVD of a matrix with respect to a weighted inner product.  We do not give the details of the reformulation here, but we do briefly recall the SVD with respect to a weighted inner product since we use this concept throughout this work.
\begin{definition}
	A \textit{core SVD} of a matrix $ P: \mathbb{R}^{n} \to \mathbb{R}_{M}^{m} $ is a decomposition $ P = V \Sigma W^T $, where $ V \in \mathbb{R}^{m \times k} $, $ \Sigma \in \mathbb{R}^{k \times k} $, and $ W \in \mathbb{R}^{n \times k} $ satisfy
	$$
	  V^T M V = I,  \quad  W^T W = I,  \quad  \Sigma = \mathrm{diag}(\sigma_1, \ldots, \sigma_k),
	$$
	where $ \sigma_1 \geq \sigma_2 \geq \cdots \geq \sigma_k > 0 $.  The values $ \{\sigma_i\} $ are called the (positive) singular values of $ P $ and the columns of $ V $ and $ W $ are called the corresponding singular vectors of $ P $.
\end{definition}
Since POD applications do not typically require the zero singular values, we do not consider the full SVD of $ P: \mathbb{R}^{n} \to \mathbb{R}_{M}^{m} $ in this work.  We do note that the SVD of $ P: \mathbb{R}^{n} \to \mathbb{R}_{M}^{m} $ is closely related to the eigenvalue decompositions of $ P^* P $ and $ P P^* $.  See \cite[Section 2.1]{Fareed17} for more details.

Also, when we consider the SVD (or core SVD) of a matrix without weighted inner products we refer to this as the standard SVD (or standard core SVD).

We consider approximately computing the SVD of a dataset $ U $ incrementally by updating the core SVD when each new column $ c $ of data is added to the data set.  This incremental procedure is performed without forming or storing the original data matrix.  Specifically, we focus on the incremental SVD algorithm with a weighted inner product proposed in Algorithm 4 of \cite{Fareed17}.  The algorithm is based on the following fundamental identity: if $ U = V \Sigma W^T $ is a core SVD, then
\begin{align*}
  [\, U \, c \, ]  &=  [\, V \Sigma W^T \, c \, ]\\
  &=  [\, V \,j \,] \left[\begin{array}{cc} \Sigma & V^{*}c\\ 0 & p \end{array}\right] \left[\begin{array}{cc} W & 0\\ 0 & 1 \end{array}\right]^T,
\end{align*}
where $ j = ( c - V V^* c )/ p $ and $  p = \| c - V V^* c  \|_{M} $ \cite{Fareed17}.  The algorithm is a modified version of Brand's incremental SVD algorithm \cite{Brand02} to directly treat the weighted inner product.  Brand's incremental SVD algorithm without a weighted inner product has been used for POD computations in \cite{ZahrFarhat15, NME:NME5283}, and our implementation strategy follows the algorithm in \cite{NME:NME5283}.

Below, we consider a slight modification of the algorithm from \cite{Fareed17}; specifically, we update the algorithm to include a computable error bound $ e $.  We show in this work that the algorithm produces the exact core SVD of a matrix $ \tilde{U} $ such that $ \| U - \tilde{U} \|_{\mathcal{L}(\mathbb{R}^s,\mathbb{R}^m_M)} \leq e $, where $ U $ is the true data matrix.  This error bound gives information about the approximation error for the singular values and singular vectors; see \Cref{subsection:error_bounds} for details.

We take the first step in the incremental SVD algorithm by initializing the SVD and the error bound with a single column $ c \neq 0 $ as follows:
$$
  \Sigma = \left\Vert \,c\,\right\Vert_{M} = (| c^{T} M c | )^{1/2},  \quad  V = c \Sigma^{-1},  \quad  W = 1, \quad e = 0.
$$
Here, the error bound $ e $ is set to zero since the initial SVD is exact.  Also, as mentioned in \cite{Fareed17}, even though $ M $ is positive definite it is possible for round off errors to cause $ c^T M c $ to be very small and negative; we use the absolute value here and throughout the algorithm to avoid this issue.

Then we incrementally update the SVD and the error bound by applying \Cref{algorithm:incrSVD_Error_weightedinner} when a new column is added.  Most of the algorithm is taken directly from \cite[Algorithm 4]{Fareed17}; we refer to that work for a detailed discussion of the algorithm and details about the implementation.

We note the following:
\begin{itemize}
	\item  The input is an existing SVD  $ V $, $ \Sigma $, and $ W $, a new column $ c $, the weight matrix $ M $, two positive tolerances, and an error bound $ e $.
	\item  Lines 10, 15, 18, 21, and 26 are new, and are simple computations used to update the error bound $ e $.
	\item  In the SVD update stage (lines 1--16), $ e_p $ is the error due to $ p $-truncation in line $ 3 $.
	\item  In the singular value truncation stage (lines 17--22), $ e_{sv} $ is the error due to the singular value truncation in line $19$.
	\item  In the orthogonalization stage (lines 23--25), a modified Gram-Schmidt algorithm with reorthogonalization is used; see Section 4.2 in \cite{Fareed17}.
	\item  The output is the updated SVD and error bound.
	\item  The columns of $ V $ are the $ M $-orthonormal POD modes, and the squares of the singular values are the POD eigenvalues.
	\item  If only the POD eigenvalues and modes are required, then the computations involving $ W $ can be skipped; however, $ W $ is needed if an approximate reconstruction of the entire data set is desired.
	\item  As new columns continue to be added, a user can monitor the computed error bound and lower the tolerances if desired.
\end{itemize}


\begin{algorithm}
	\caption{Incremental SVD and error bound with weighted inner product }
	\label{algorithm:incrSVD_Error_weightedinner}
	\begin{algorithmic}[1]
		\REQUIRE  $ V \in \mathbb{R}^{m \times k} $, $ \Sigma \in \mathbb{R}^{k \times k} $, $ W \in \mathbb{R}^{n \times k}$, $ c \in \mathbb{R}^m $, $ M \in \mathbb{R}^{m \times m} $, $ \mathrm{tol} $, $ \mathrm{tol}_\mathrm{sv} $, $ e $
		\\
		\smallskip
		\% Prepare for SVD update
		\smallskip
		\STATE  $d=V^{T} M c$, $ p = \mathrm{sqrt}( |( c-Vd )^{T} M ( c-Vd )| )$ 
		\IF{$(p < \mathrm{tol})$} 
		\STATE  $ Q = \begin{bmatrix}
		\Sigma & d \\
		0 & 0 \end{bmatrix} $
		\ELSE
		\STATE  $ Q = \begin{bmatrix}
		\Sigma & d \\
		0 & p
		\end{bmatrix} $
		\ENDIF
		\STATE  $[\,V_{Q},\Sigma_{Q},W_{Q}\,]=\mathrm{svd}(Q)$ 
		\\
		\smallskip
		\% SVD update
		\smallskip
		\IF{$( p < \mathrm{tol} )$ or $( k \geq  m )$}
		\STATE $ V = V V_{Q_{(1:k,1:k)}}$, $\Sigma = \Sigma_{Q_{(1:k,1:k)}}$, $ W =\begin{bmatrix} W & 0\\ 0 & 1 \end{bmatrix} W_{Q_{(:,1:k)}} $
		\STATE $ e_{p} = p $
		\ELSE
		\STATE $ j=(c-Vd) / p $
		\STATE $ V= [V \, j] V_{Q} $, $ \Sigma = \Sigma_{Q}$,  $ W = \begin{bmatrix}
		W & 0\\
		0 & 1
		\end{bmatrix} W_{Q} $
		\STATE $ k=k+1 $
		\STATE $ e_{p} = 0 $
		\ENDIF
		\\
		\smallskip
		\% Neglect small singular values: truncation
		\smallskip
		\IF{$(\Sigma_{(r,r)}  > \mathrm{tol}_\mathrm{sv})$ and $(\Sigma_{(r+1,r+1)} \leq \mathrm{tol}_\mathrm{sv})$}     
        \STATE $ e_{sv} = \Sigma_{(r+1,r+1)} $			
        \STATE $ \Sigma = \Sigma_{(1:r,1:r)}$, \quad  $V = V_{(:,1:r)}$, \quad  $W = W_{(:,1:r)}$  
		\ELSE
		\STATE $ e_{sv} = 0 $
		\ENDIF
		\\
		\smallskip
		\% Orthogonalize if necessary
		\smallskip
		\IF{( $ | V_{(:,\mathrm{end})}^T M V_{(:,1)} | > \mathrm{min}(\mathrm{tol} , \mathrm{tol}\times m) $)}
		\STATE $ V = \mathrm{modifiedGSweighted}(V,M) $   
        \ENDIF		
        \STATE $ e = e + e_{p} + e_{sv} $
     	\RETURN  $ V $, $ \Sigma $, $ W $, $ e $
	\end{algorithmic}
\end{algorithm}

\section{Error Analysis}
\label{subsec:Error Analysis}

In this section, we perform an error analysis of \Cref{algorithm:incrSVD_Error_weightedinner}.  We show the algorithm produces the exact SVD of another matrix $ \tilde{U} $, and bound the error between the matrices.

We assume all computations in the algorithm are performed in exact arithmetic.  Therefore, the Gram-Schmidt orthogonalization stage (in lines 23--25) is not considered here.  We note that in \cite{Fareed17}, we considered a Gram-Schmidt procedure with reorthogonalization to minimize the effect of round-off errors; see, e.g., \cite{GiraudLangou02,GiraudLangou05,GiraudLangouRozloznik05,RozloznikTuma12}.  We leave an analysis of round-off errors in \Cref{algorithm:incrSVD_Error_weightedinner} to be considered elsewhere.

We begin our analysis in \Cref{subsec:ind_trunc_errors} by analyzing the error due to each individual truncation step in the algorithm.  Then we provide error bounds for the algorithm in \Cref{subsection:error_bounds}.

\subsection{Individual Truncation Errors}
\label{subsec:ind_trunc_errors}
%

We begin our analysis of the incremental SVD algorithm by recalling a result from \cite{Fareed17}.  This result shows that a single column incremental update to the SVD is exact without truncation when $ p = \| c-VV^{*}c \|_{M} > 0 $.

\begin{theorem}[Theorem 4.1 in \cite{Fareed17}]\label{thm:svd_exact_update}
Let $ U: \mathbb{R}^{n}\longrightarrow \mathbb{R}_{M}^{m}$, and suppose $ U = V \Sigma W^{T}$ is an exact core SVD of $ U $, where $ V^{T} M V = I $ for $V \in \mathbb{R}^{m \times k}$, $ W^{T} W = I$ for $W \in \mathbb{R}^{n \times k}$, and $\Sigma \in \mathbb{R}^{k \times k}$.  Let $ c \in \mathbb{R}^{m}_M $ and define
\[
  h = c-VV^{*}c,  \quad  p = \| h \|_{M},  \quad
  Q=\begin{bmatrix}
  ~\Sigma & V^{*}c\\
  0 & p
  \end{bmatrix},
\]
where $ V^* = V^T M $.  If $ p > 0 $ and a standard core SVD of $ Q \in \mathbb{R}^{k+1 \times k+1} $ is given by
 \begin{equation}\label{eq:4}
  Q =V _{Q}\,\Sigma_{Q}\,W_{Q}^{T},
\end{equation}
then a core SVD of $ [ \,U \,\, c \, ] : \mathbb{R}^{n+1}\longrightarrow \mathbb{R}_{M}^{m} $ is given by
\[
 [\,U \,\, c \,] = V_{u} \Sigma_{Q} W_{u}^{T},
\]
where
\[
 V_{u} = [\, V \,\, j \,] ~V_{Q},  \quad  j = h/p,  \quad  W_{u} = \left[\begin{array}{cc}
                                      W & 0 \\
                                      0 & 1
                                    \end{array}\right] W_{Q}.
\]
\end{theorem}

Next, we analyze the incremental SVD update in the case when the added column $ c $ satisfies $ p = \| c - V V^* c \|_M = 0 $.
\begin{lemma} \label{prop:2.2}
	Let $ U = V \Sigma W^{T} $, $ c $, $ h $, $ p $, and $ Q $ be given as in \Cref{thm:svd_exact_update}, and assume $ p = \| c - V V^* c \|_M = 0 $.  If the full standard SVD of $ Q \in \mathbb{R}^{k+1 \times k+1} $ is given by $ Q = V_Q \Sigma_Q W_Q^T $, where $ V_Q, \Sigma_Q, W_Q \in \mathbb{R}^{k+1 \times k+1} $, then
	\begin{equation*}
	V_Q = \begin{bmatrix}
	V_{Q_{(1:k,1:k)}} & 0 \\  0 & 1 \end{bmatrix}, \quad   \Sigma_Q = \begin{bmatrix}  \Sigma_{Q_{(1:k,1:k)}} & 0 \\ 0 & 0  \end{bmatrix},  \quad  \Sigma_{Q_{(1:k,1:k)}} > 0,
	\end{equation*}
	and a standard core SVD of $ R = Q_{(1:k,1:k+1)} = [ \, \Sigma \,\,\, V^* c \,] \in \mathbb{R}^{k \times k+1} $ is given by
	\[
	  R = V_{Q_{(1:k,1:k)}} \Sigma_{Q_{(1:k,1:k)}} (W_{Q_{(:,1:k)}})^T.
	\]
%
\end{lemma}
\begin{proof}
Let $ \sigma_{Q_1} \geq \sigma_{Q_2} \geq \cdots \geq \sigma_{Q_{k+1}} \geq 0 $ be the singular values of $ Q $ so that $ \Sigma_Q = \mathrm{diag}(\sigma_{Q_{1}}, ..., \sigma_{Q_{k+1}})$.  Also, let $ \{ v_{Q_j} \} $ and $ \{ w_{Q_j} \} $ be the corresponding orthonormal singular vectors in $ \mathbb{R}^{k+1} $, so that
$$
  V_Q = [ v_{Q_{1}}, \ldots, v_{Q_{(k+1)}} ],  \quad  W_Q = [ w_{Q_{1}}, \ldots, w_{Q_{(k+1)}} ],
$$
with $ V_Q^{T} V_Q = I $ and $ W_Q^{T} W_Q = I $.

First, we show $ Q $ has exactly one zero singular value.  Since we know
\begin{align}
   Q^{T} v_{Q_j} &= \sigma_{Q_j} w_{Q_j},\label{Q_SVD_1} \\
   Q w_{Q_j} &= \sigma_{Q_j} v_{Q_j}, \label{Q_SVD_2}
\end{align}
for $ j = 1, \ldots, k+1 $, the number of zero singular values of $ Q $ is precisely equal to the dimension of the nullspace of $ Q^T $.  Suppose $ v = [v_1, \ldots, v_{k+1}]^T \in \mathbb{R}^{k+1} $ satisfies $ Q^T v = 0 $.  Recall $ \Sigma = \mathrm{diag}(\sigma_1,\sigma_2, \ldots, \sigma_k) > 0 $, and let $ d = V^* c = [d_1, \ldots, d_k]^T $.  Then $ Q^T v = 0 $ implies
$$
\begin{bmatrix}
                 \sigma_{1} v_1 \\
                 \sigma_{2} v_2 \\
                 \vdots \\
                 \sigma_{k} v_k \\
                 d_{1} v_1 + d_{2} v_2 + \ldots + d_{k} v_k
                \end{bmatrix} = \begin{bmatrix}
                                    0 \\
                                    0 \\
                                \vdots \\
                                    0 \\
                                    0
                                  \end{bmatrix}.
$$
Since $ \sigma_{1}\geq \cdots \geq \sigma_{k} > 0 $, we have $ v_j = 0 $ for $ j = 1,\ldots, k $.  This implies the nullspace of $ Q^T $ is exactly the span of $ e_{k+1} = [0, \ldots, 0, 1]^T \in \mathbb{R}^{k+1} $.  Therefore, the nullspace is one dimensional and $ Q $ has exactly one zero singular value, i.e., $ \sigma_{Q_{k+1}} = 0 $ and  $ \sigma_{Q_1} \geq \sigma_{Q_2} \geq \cdots \geq \sigma_{Q_{k}} > 0 $.

Next, $ Q w_{Q_{j}} = \sigma_{j} v_{Q_{j}} $ for $ j = 1, \ldots, k $ gives
$$
   \Rightarrow \begin{bmatrix}
                 \sigma_{1} w_{Q_{j,1}} + d_{1} w_{Q_{j,k+1}} \\
                 \sigma_{2} w_{Q_{j,2}} + d_{2} w_{Q_{j,k+1}} \\
                 \vdots \\
                 \sigma_{k} w_{Q_{j,k}} + d_{k} w_{Q_{j,k+1}} \\
                 0
               \end{bmatrix} = \begin{bmatrix}
                 \sigma_{j} v_{Q_{j,1}} \\
                 \sigma_{j} v_{Q_{j,2}} \\
                 \vdots \\
                 \sigma_{j} v_{Q_{j,k}} \\
                 \sigma_{j} v_{Q_{j,k+1}}
               \end{bmatrix}.
$$
The last equation gives $ v_{Q_{j,k+1}} = 0 $ since $ \sigma_{j} > 0 $ for $ j = 1, \ldots , k $.  Therefore, for $ j = 1, \ldots , k $,
$$
  v_{Q_{j}} = [ v_{Q_{j,1}}, v_{Q_{j,2}}, \ldots, v_{Q_{j,k}}, 0 ]^T,
$$
and
$$
  v_{Q_{k+1}} = [ 0, 0, \ldots, 0, 1 ]^T.
$$
This implies
$$
V_Q = \begin{bmatrix}
      V_{Q_{(1:k,1:k)}} & 0 \\
      0 & 1
      \end{bmatrix},
$$
and so the SVD decomposition of Q is given by
\begin{align*}
  Q  &=  \begin{bmatrix}
      V_{Q_{(1:k,1:k)}} & 0 \\
      0 & 1
    \end{bmatrix} \begin{bmatrix} \Sigma_{Q_{(1:k,1:k)}} & 0 \\ 0 & 0 \end{bmatrix} W_{Q}^T.
\end{align*}

This gives $ R = Q_{(1:k,1:k+1)} = \check{V}_{Q} \check{\Sigma}_{Q} \check{W}_{Q}^T $, where $ \check{V}_{Q} = V_{Q_{(1:k,1:k)}} $, $ \check{\Sigma}_{Q} = \Sigma_{Q_{(1:k,1:k)}} $, and $ \check{W}_{Q} = W_{Q_{(1:k+1,1:k)}} $.  It can be checked that $ \check{V}_{Q}^{T} \check{V}_{Q} = I $ and $ \check{W}_{Q}^{T} \check{W}_{Q} = I $ since $ V_Q^{T} V_Q = I $ and $ W_Q^{T} W_Q = I $.  Therefore, a standard core SVD of $ R \in \mathbb{R}^{k \times k+1} $ is given by $ R = \check{V}_{Q} \check{\Sigma}_{Q} \check{W}_{Q}^T $.

\end{proof}

The following result is nearly identical to Proposition 2.3 in \cite{Fareed17}; the proof is also almost identical and is omitted.
\begin{lemma}[Proposition 2.3 in \cite{Fareed17}]\label{prop:2.3}
  Suppose $ V_u \in \mathbb{R}^{m \times k} $ has $ M $-orthonormal columns and $ W_u \in \mathbb{R}^{n \times l} $ has orthonormal columns.  If $ R \in \mathbb{R}^{k \times l} $ has standard core SVD $ R = V_{R} \Sigma_{R} W_{R}^T $ and $ P: \mathbb{R}^{n} \to \mathbb{R}_{M}^{m}$ is defined by $ P = V R W^T $, then
  \begin{equation}\label{eqn:P_coreSVD_product}
    P = V_u \Sigma_u W_u^T,  \quad  V_u = V V_{R},  \quad  \Sigma_u = \Sigma_R,  \quad  W_u = W W_{R},
  \end{equation}
  is a core SVD of $ P $.
\end{lemma}

Next, we complete the analysis of the $ p = 0 $ case:
\begin{proposition} \label{prop:p_zero}
	Let $ U = V \Sigma W^{T} $, $ c $, $ h $, $ p $, and $ Q $ be given as in \Cref{thm:svd_exact_update}, and assume $ p = \| c - V V^* c \|_M = 0 $.  If the full standard SVD of $ Q \in \mathbb{R}^{k+1 \times k+1} $ is given by $ Q = V_Q \Sigma_Q W_Q^T $, where $ V_Q, \Sigma_Q, W_Q \in \mathbb{R}^{k+1 \times k+1} $, then a core SVD of $ [\, U \,\, c \, ] : \mathbb{R}^{n+1} \to \mathbb{R}^m_M $ is given by
	$$
	  [\,U \,\, c \,] = V_{u} \Sigma_u W_{u}^{T},
	$$
	where
	$$
	  V_{u} = V V_{Q_{(1:k,1:k)}},  \quad  \Sigma_u = \Sigma_{Q_{(1:k,1:k)}},  \quad  W_{u} = \left[\begin{array}{cc}
	  W & 0 \\
	  0 & 1
	  \end{array}\right] W_{Q_{(:,1:k)}}.
	$$
\end{proposition}
\begin{proof}
Since $ p = 0 $, we have $ c = V V^* c $ and therefore
$$
    \left[\begin{array}{cc} U & c \end{array}\right]  =  \left[\begin{array}{cc}  V \Sigma W^T & V V^* c \end{array}\right]  =  V \left[\begin{array}{cc} \Sigma & V^{*}c \end{array}\right] \left[\begin{array}{cc} W & 0\\ 0 & 1 \end{array}\right]^T.
$$
The result follows from \Cref{prop:2.2} and \Cref{prop:2.3} by taking $ P = [ \, U \,\, c \, ] $ and $ R = [ \, \Sigma \,\, V^* c \, ] $.
\end{proof}

\textbf{Truncation part 1.}  Next, we analyze the incremental SVD update in the case when the added column $ c $ satisfies $ p = \| c - V V^* c \|_M < \mathrm{tol} $.  In this case, \Cref{algorithm:incrSVD_Error_weightedinner} does not compute the SVD of $ [ \, U \,\, c \, ] $.  Instead, \Cref{algorithm:incrSVD_Error_weightedinner} sets $ p = 0 $ and returns the exact SVD of $ \tilde{U} = [ \, U \,\,\, V V^* c \,] $.  The approximation error in the operator norm is given in the next result.
\begin{proposition} \label{trn_1}
  Let $ U: \mathbb{R}^{n}\longrightarrow \mathbb{R}_{M}^{m}$, and suppose $ U = V \Sigma W^{T} $ is a core SVD of U. If $ c \in \mathbb{R}^m_M $, $ p = \| c - VV^{*}c \|_{M} $, and
$$
\tilde{U} = [ \, U \,\, VV^{*}c\,],
$$
then
$$
\| [ \,U \,\, c\,] - \tilde{U} \|_{ \mathcal{L}( \mathbb{R}^{n+1}, \mathbb{R}_{M}^{m})}  = p.
$$
\end{proposition}
\begin{proof}
For $ x = [x_1, \ldots, x_{n+1} ]^T \in \mathbb{R}^{n+1} $, we have
\begin{align*}
  \| [ \,U \,\, c\,] - \tilde{U} \|_{ \mathcal{L}( \mathbb{R}^{n+1}, \mathbb{R}_{M}^{m})} & = \sup_{\|x\|=1} \big\|  [ \,0 \,\,\, (c-VV^{*}c) \,] x \big\|_{M}\\
                                                                                      & = \, \sup_{\|x\|=1} \| c - VV^{*}c \|_{M}  \,  | x_{n+1} |\\
                                                                                      & = \| c - VV^{*}c \|_{M},
\end{align*}
where the $ \sup $ is clearly attained by $ x = [0, \ldots, 0, 1 ]^T \in \mathbb{R}^{n+1} $.
\end{proof}

\textbf{Truncation part 2.}  In \Cref{algorithm:incrSVD_Error_weightedinner}, after the SVD update due to an added column the algorithm truncates any singular values that are smaller than a given tolerance, $ \mathrm{tol}_\mathrm{sv} $.  For the matrix case with unweighted inner products, the operator norm error caused by this truncation is well-known to equal the first neglected singular value.  This result is also true for a compact linear operator mapping between two Hilbert spaces; see, e.g., \cite[Chapters VI--VIII]{GohbergGoldbergKaashoek90}, \cite[Chapter 30]{Lax02}, and \cite[Sections VI.5--VI.6]{ReedSimon80} for more information about the SVD for compact operators.  This gives the following result:
\begin{proposition} \label{trn_2}
  Let $ U: \mathbb{R}^{n}\longrightarrow \mathbb{R}_{M}^{m}$, and suppose $ U = V \Sigma W^{T} $ is a core SVD of U.  For a given $ r > 0 $, let $ \tilde{U} $ be the rank $ r $ truncated SVD of $ U $, i.e.,
$$
  \tilde{U} = V_{(:,1:r)} \Sigma_{(1:r,1:r)} (W_{(:,1:r)})^{T}.
$$
Then
$$
  \| U - \tilde{U} \|_{ \mathcal{L}( \mathbb{R}^{n}, \mathbb{R}_{M}^{m})} = \Sigma_{(r+1,r+1)}.
$$
\end{proposition}
%

\subsection{Error Bounds}
\label{subsection:error_bounds}

Next, we fully explain the computed error bound in \Cref{algorithm:incrSVD_Error_weightedinner}.  In a typical application of the algorithm, many new columns of data are added and the POD is updated many times.  In the following result, we assume we are at the $ k $th step of this procedure and we have an existing error bound.  We prove that \Cref{algorithm:incrSVD_Error_weightedinner} produces a correct update of the error bound.

More specifically, let $ k \in \mathbb{N} $, let $ U_k, \tilde{U}_k : \mathbb{R}^k \to \mathbb{R}^m_M $, and assume
$$
U_k = V_k \Sigma_k W_k^{T},  \quad  \tilde{U}_k = \tilde{V}_k \tilde{\Sigma}_k \tilde{W}_k^{T}
$$
are core SVDs of $ U $ and $ \tilde{U} $.  Let $ c_k \in \mathbb{R}^m_M $ and define $ U_{k+1} := [ U_k \,\,c_k ] : \mathbb{R}^{k+1} \to \mathbb{R}^m_M $.  Furthermore, let $ \tilde{U}_{k+1} : \mathbb{R}^{k+1} \to \mathbb{R}^m_M $ be the result of one step of the incremental SVD update applied to $ \tilde{U}_{k} $ so that
$$
  \tilde{U}_{k+1} = \tilde{V}_{k+1} \tilde{\Sigma}_{k+1} \tilde{W}_{k+1}^{T}.
$$
Therefore, we consider the sequence $ \{ U_k \} $ to be the exact data matrices, and the sequence $ \{ \tilde{U}_k \} $ to be the result produced (in exact arithmetic) by \Cref{algorithm:incrSVD_Error_weightedinner}.

In exact arithmetic, there are two stages to \Cref{algorithm:incrSVD_Error_weightedinner}.  The first stage is the SVD update in lines 1--16.  This stage of the algorithm takes $ \tilde{U}_k $ and the added column $ c $ and produces the update $ \hat{U}_{k+1} $.  There are two possible results for $ \hat{U}_{k+1} $ depending on the value of $ p $ in line 1.  The second stage is the singular value truncation applied to $ \hat{U}_{k+1} $ (lines 17--22), which produces the final update $ \tilde{U}_{k+1} $.  Again, there are two possible results for $ \tilde{U}_{k+1} $, depending on the singular values of $ \hat{U}_{k+1} $.  We analyze the error bound for each possible outcome of the algorithm in the result below.

Let the positive tolerances $ \mathrm{tol} $ and $ \mathrm{tol}_\mathrm{sv} $ be fixed.  Below, we let $ p_k $ denote the value $ p $ in line 1 of \Cref{algorithm:incrSVD_Error_weightedinner}.  We say that $ p $ truncation is applied if $ p_k < \mathrm{tol} $.  We say the singular value truncation is applied if any of the singular values of $ \hat{U}_{k+1} $ are less than $ \mathrm{tol}_\mathrm{sv} $.  In this case, we find a value $ r $ so that the first $ r $ largest singular values of $ \hat{U}_{k+1} $ are greater than $ \mathrm{tol}_\mathrm{sv} $, while the remaining singular values are less than or equal to $ \mathrm{tol}_\mathrm{sv} $.  We let $ \hat{\sigma}_{r+1} $ denote the largest singular value of $ \hat{U}_{k+1} $ such that $ \hat{\sigma}_{r+1} \leq \mathrm{tol}_\mathrm{sv} $.
\begin{theorem} \label{Lemma:1step_errorbound}
%
If
$$
  \| U_k - \tilde{U}_k \|_{\mathcal{L}(\mathbb{R}^k,\mathbb{R}_M^m)} \leq e_k,  \quad  p_k = \| c_k - \tilde{V}_k \tilde{V}_k^* c_k \|_{M},
$$
then
$$
  \| U_{k+1} - \tilde{U}_{k+1} \|_{\mathcal{L}(\mathbb{R}^{k+1},\mathbb{R}_M^m)} \leq e_{k+1},
$$
where
\begin{equation*}
  e_{k+1}  =  \begin{cases}
			  e_k,  &  \text{if no truncation is applied,}\\
			  e_k + p_k,  &  \text{if only $ p $ truncation is applied,}\\
			  e_k + \hat{\sigma}_{r+1},  &    \text{if only the singular value truncation is applied,}\\
			  e_k + p_k + \hat{\sigma}_{r+1},  &  \text{if both truncations are applied.}
 			  \end{cases}
\end{equation*}
\end{theorem}
\begin{proof}
Stage 1 of \Cref{algorithm:incrSVD_Error_weightedinner} (lines 1--16) takes $ \tilde{U}_k $ and produces $ \hat{U}_{k+1} $.  If $ p_k \geq \mathrm{tol} $, then \Cref{thm:svd_exact_update} gives that the core SVD is updated exactly, i.e.,
$$
  \hat{U}_{k+1} = [ \, \tilde{U}_k \, \, c_k \, ]  \quad  \mbox{if $ p_k \geq \mathrm{tol} $.}
$$
Otherwise, if $ p_k < \mathrm{tol} $, then \Cref{trn_1} implies
$$
  \hat{U}_{k+1} = [ \, \tilde{U}_k \, \, \tilde{V} \tilde{V}^* c_k \, ]  \quad  \mbox{if $ p_k < \mathrm{tol} $,}
$$
and the error is given by
$$
  \| [ \, \tilde{U}_k \, \, c_k \, ] - \hat{U}_{k+1} \|_{\mathcal{L}(\mathbb{R}^{k+1},\mathbb{R}_M^m)} = p_k.
$$

Stage 2 of \Cref{algorithm:incrSVD_Error_weightedinner} (lines 17--22) takes $ \hat{U}_{k+1} $ and produces $ \tilde{U}_{k+1} $.  If all of the singular values of $ \hat{U}_{k+1} $ are greater than $ \mathrm{tol}_\mathrm{sv} $, then $ \hat{U}_{k+1} = \tilde{U}_{k+1} $ and there is no error in this stage.  Otherwise, let $ \hat{\sigma}_{r+1} $ denote the largest singular value of $ \hat{U}_{k+1} $ such that $ \hat{\sigma}_{r+1} \leq \mathrm{tol}_\mathrm{sv} $.  In this case, $ \tilde{U}_{k+1} $ is simply the $ r $th order truncated SVD of $ \hat{U}_{k+1} $, and the error is given by \Cref{trn_2}:
$$
  \| \tilde{U}_{k+1} - \hat{U}_{k+1} \|_{\mathcal{L}(\mathbb{R}^{k+1},\mathbb{R}_M^m)} = \hat{\sigma}_{r+1}.
$$

Below, for ease of notation, let $ \| \cdot \| $ denote the $ \mathcal{L}(\mathbb{R}^{k+1},\mathbb{R}_M^m) $ operator norm.  The error between $ U_{k+1} $ and $ \tilde{U}_{k+1} $ in the operator norm can be bounded as follows:
$$
  \| U_{k+1} - \tilde{U}_{k+1} \|  \leq  \| U_{k+1} - [ \, \tilde{U}_k \, \, c_k \,] \| + \| [ \, \tilde{U}_k \, \, c_k \,] - \hat{U}_{k+1} \| + \| \hat{U}_{k+1} - \tilde{U}_{k+1} \|.
$$
As noted above, the second error term is either zero if $ p $ truncation is not applied or $ p_k $ otherwise.  Also, the third error term is either zero if the singular values truncation is not applied or $ \hat{\sigma}_{r+1} $ otherwise.  For the first term, we have
\begin{align*}
  \| U_{k+1} - [ \, \tilde{U}_k \, \, c_k \,] \|  &=  \| [ \, U_k \, \, c_k \,] - [ \, \tilde{U}_k \, \, c_k \,] \|\\
    &=  \|  ( U_k - \tilde{U}_k) \, \, 0 \|\\
    &=  \sup_{ \| x \| = 1 }  \big\| [  ( U_k - \tilde{U}_k) \, \, 0  ]  x \big\|_{M}\\
    &\leq  \| U_k - \tilde{U}_k \|_{ \mathcal{L}(\mathbb{R}^{k},\mathbb{R}_M^m) } \leq e_k.
\end{align*}
This completes the proof.
\end{proof}

The result above explains the update of the error bound in one step of \Cref{algorithm:incrSVD_Error_weightedinner}.  Now we assume the SVD is initialized exactly when $ k = 1 $, and then the algorithm is applied for a sequence of added columns $ \{ c_k \} \subset \mathbb{R}_M^m $, for $ k = 2, \ldots, s $.
\begin{corollary}
  Let $ \mathrm{tol} $ and $ \mathrm{tol}_\mathrm{sv} $ be fixed positive constants, and let $ \{ c_k \} \subset \mathbb{R}_M^m $, for $ k = 1, \ldots, s $, be the columns of a matrix $ U $.  For $ k = 1 $, assume the SVD $ \tilde{U}_1 = \tilde{V}_1 \tilde{\Sigma}_1 \tilde{W}_1^T $ and error bound $ e_1 = 0 $ are initialized exactly as described in \Cref{Section:basic}.  For $ k = 1, \ldots, s-1 $, let $ \tilde{U}_{k+1} = \tilde{V}_{k+1} \tilde{\Sigma}_{k+1} \tilde{W}_{k+1}^T $ and $ e_{k+1} $ be the output of \Cref{algorithm:incrSVD_Error_weightedinner} applied to the input $ \tilde{U}_k = \tilde{V}_k \tilde{\Sigma}_k \tilde{W}_k^T $ and $ e_k $.  If $ T_p $ represents the total number of times $ p $ truncation is applied and $ T_\mathrm{sv} $ represents the total number of times the singular value truncation is applied, then
  \[
    \|  U - \tilde{V}_{s} \tilde{\Sigma}_{s} \tilde{W}_{s} \|_{ \mathcal{L}(\mathbb{R}^{s},\mathbb{R}_M^m) }  \leq  T_{p} \mathrm{tol} + T_\mathrm{sv} \mathrm{tol}_\mathrm{sv}.
  \]
\end{corollary}
\begin{proof}
  The proof follows immediately from the previous result, using $ p_k \leq \mathrm{tol} $ and $ \hat{\sigma}_{r+1} \leq \mathrm{tol}_\mathrm{sv} $.
 \end{proof}
The error bound in the result above is not as precise as the error bound computed using \Cref{algorithm:incrSVD_Error_weightedinner} since the tolerances are only upper bounds on the errors in each step.  However, this result does provide some insight into the choice of the tolerances for the algorithm.  Specifically, in general there is no reason to expect one of $ T_p $ or $ T_\mathrm{sv} $ to be significantly larger than the other; therefore, it seems reasonable to choose equal values for the tolerances.  Furthermore, for a very large number of added columns, it is possible that $ T_p $ and $ T_\mathrm{sv} $ can be large; therefore, small tolerances should be chosen to preserve accuracy.

\Cref{algorithm:incrSVD_Error_weightedinner} computes an upper bound on the operator norm error between the exact data matrix $ U $ and the approximate truncated SVD $ \tilde{U} = \tilde{V} \tilde{\Sigma} \tilde{W}^T $ of the data matrix.  (The above corollary also provides another upper bound on the error.)  This error bound allows us to bound the error in the incrementally computed singular values and singular vectors.  Let $\{\sigma_k, v_k, w_k \}_{k \geq 1}$ and $\{\tilde{\sigma}_k, \tilde{v}_k, \tilde{w}_k \}_{k \geq 1}$ denote the ordered singular values and corresponding orthonormal singular vectors of $ U, \tilde{U} : \mathbb{R}^s \to \mathrm{R}^m_M $ in the result below.  The following result follows directly from general results about error bounds for singular values and singular vectors of compact linear operators in \Cref{Appen}.
\begin{theorem}
	Let $ k \geq 1 $, and let $ \varepsilon > 0 $ such that $ \| U - \tilde{U} \|_{ \mathcal{L}( \mathbb{R}^s, \mathbb{R}^m_M)} \leq \varepsilon $.  Then
	$$
	  | \sigma_\ell - \tilde{\sigma}_\ell | \leq \varepsilon  \quad  \mbox{for all $ \ell \geq 1 $.}
	$$
	Also, for $ j = 1, \ldots, k $, define
	$$
	\varepsilon_j = j \varepsilon + 2 \sum_{i=1}^{j-1} \left( \varepsilon_i + \sigma_{i} E_i^{1/2} \right),  \quad  E_j = 2 \left( 1 - \sqrt{ \frac{ ( \sigma_j - 2 \varepsilon_j )^2 - \sigma_{j+1}^2 }{ \sigma_j^2 - \sigma_{j+1}^2 } } \right).
	$$
	If the first $ k+1 $ singular values of $ U $ are distinct and positive, the singular vector pairs $ \{ \tilde{v}_j, \tilde{w}_j \}_{j=1}^k $ are suitably normalized, and
	$$
	\varepsilon_j \leq \frac{ \sigma_j - \sigma_{j+1} }{ 2 }  \quad  \mbox{for $ j = 1, \ldots, k $,}
	$$
	then
	\begin{equation}\label{diff_kthsvectors}
	\| v_j - \tilde{v}_j \|_M  \leq  E_j^{1/2},  \quad  \| w_j - \tilde{w}_j \|  \leq  E_j^{1/2} + 2 \sigma_j^{-1} \varepsilon_j,  \quad  \mbox{for $ j = 1, \ldots, k $.}	
	\end{equation}
\end{theorem}
This result indicates we should expect accurate approximate singular values and also accurate approximate singular vectors if $ \varepsilon $ is small and there is not a small gap in the singular values.  We note that POD singular values often decay to zero quickly, and therefore we expect to see lower accuracy in the computed POD modes for smaller singular values due to the small gap.  The examples in our first work \cite{Fareed17} and the new examples below show both of these expected behaviors for the errors in the approximate singular vectors.

\section{Numerical Results}
\label{sec:numerical_results}

We consider the 1D FitzHugh-Nagumo system 
\begin{align*}
  \frac{\partial v(t,x)}{\partial t} &= \mu \frac{\partial^{2}v(t,x)}{\partial x^{2}} - \frac{1}{\mu} w(t,x) +  \frac{1}{\mu} f(v) + \frac{c}{\mu}, \quad  0<x<1,\\
  \frac{\partial w(t,x)}{\partial t} &= b v(t,x) - \gamma w(t,x) + c, \quad  0<x<1,
\end{align*}
where $f(v)  = v(v - 0.1)(1 - v)$, $\mu = 0.015 $, $ b = 0.5 $, $ \gamma = 2 $, $ c = 0.05 $, the boundary conditions are $ v_x(t,0) = - 50000 t^{3} e^{-15t} $, $ v_x(t,1) = 0 $, and the initial conditions are zero.  This example problem was considered in \cite{Wang15}, and we used the interpolated coefficient finite element method from that work to discretize the problem in space.  For the finite element method we used continuous piecewise linear basis functions with equally spaced nodes, and we used Matlab's \texttt{ode23s} to approximate the solution of the resulting nonlinear ODE system on different time intervals.

For the POD computations, we consider the data $ z(t,x) = [ v(t,x), w(t,x) ] $ in the Hilbert space $ L^2(0,1) \times L^2(0,1) $ with standard inner product.  Now we follow the procedure in our first work \cite{Fareed17} to arrive at the weighted SVD problem.  At each time step, we rescale the approximate solution data by the square root of the time step; see \cite[Section 5.1]{Fareed17}.  We expand the approximate solution in the finite element basis to obtain the weight matrix $ M $ as in \cite[Section 5.2]{Fareed17}.  To compute the POD of the approximate solution data, we compute the SVD of the finite element solution coefficient matrix $ U : \mathbb{R}^s \to \mathbb{R}^m_M $, where $ s $ is the number of time steps (snapshots) and $ m $ is two times the number of finite element nodes.

To illustrate our analysis of the incremental SVD algorithm, we consider three examples:
\begin{description}
	\item[Example 1]  $5000$ finite element nodes and  $ s = 491 $ snapshots in the time interval $[ 0 , 10 ]$
	\item[Example 2]  $10000$ finite element nodes and  $ s = 710 $ snapshots in the time interval $[ 0 , 15 ]$
	\item[Example 3]  $50000$ finite element nodes and  $ s = 1275 $ snapshots in the time interval $[ 0 , 28 ]$
\end{description}
We consider relatively small values of $ m = 2 \times \text{nodes} $ and $ s $ in order to test the incremental algorithm against exact SVD computations.

Let $ U $ denote the finite element solution coefficient matrix, and let $ \tilde{U} = \tilde{V} \tilde{\Sigma} \tilde{W}^T $ denote the incrementally computed approximate SVD of $ U : \mathbb{R}^s \to \mathbb{R}^m_M $ produced by \Cref{algorithm:incrSVD_Error_weightedinner}.  For each example, we choose various tolerances and compute:
\begin{gather*}
  \mbox{Rank $ = \mathrm{rank}(\tilde{U}) $,}  \quad  \mbox{Exact error $ = \| U - \tilde{U} \|_{ \mathcal{L}(\mathbb{R}^s,\mathbb{R}^m_M) } $,}\\
  \mbox{Incr.\ error bound $ = e $ computed by \Cref{algorithm:incrSVD_Error_weightedinner} at the final snapshot.}
\end{gather*}
The exact SVD of $ U : \mathbb{R}^s \to \mathbb{R}^m_M $ and the exact error are both computed using a Cholesky factorization of the weight matrix $ M $ following Algorithm 1 in \cite{Fareed17}.  The exact computations are for testing only since they require storing all of the data.

\Cref{table:Error_SVD ex1}--\Cref{table:Error_SVD_ex3} display the computed quantities listed above for the three examples with various choices of the $ p $ truncation tolerance, $ \mathrm{tol}$, and the singular value truncation tolerance, $ \mathrm{tol}_{\mathrm{sv}} $.  We set each tolerance to $ 10^{-8} $, $ 10^{-10} $, or $ 10^{-12} $, for a total of nine tests for each example.  In all of the tests, the incrementally computed error bound is larger than the exact error and the error bound is small.  Also, the tests indicate that there is no benefit from choosing one tolerance different than the other.
\begin{table}
	\renewcommand{\arraystretch}{1.25}
	\centering
	\begin{tabular}{llccc}
		\hline
		$  \mathrm{tol} $   &   $ \mathrm{tol}_\mathrm{sv}$  &   Rank    &   Exact error &   Incr.\ error bound   \\ \hline
		$ 10^{-8} $ &   $ 10^{-8} $ &   $36$    &   $ 3.6924e-07 $ &    $2.8029e-06 $\\ 
		$ 10^{-8} $ &   $ 10^{-10}$ &   $66$    &   $3.1932e-07 $  &    $1.1826e-06 $\\ 
		$ 10^{-8} $ &   $ 10^{-12}$ &   $61$    &  $8.5938e-07 $   &    $9.0495e-07 $\\ 
		$ 10^{-10}$ &   $ 10^{-8} $ &   $30$    &   $3.9090e-08 $  &    $1.4908e-06 $\\ 
		$ 10^{-10}$ &   $ 10^{-10}$ &   $44 $   &  $4.4893e-10 $   &    $2.7417e-08 $\\ 
		$ 10^{-10}$ &   $ 10^{-12}$ &   $71$    &   $3.9349e-10 $  &    $8.9680e-09 $\\ 
		$ 10^{-12}$ &  $ 10^{-8} $  &   $30$    &   $3.9090e-08 $  &    $1.4908e-06 $\\ 
		$ 10^{-12}$ &  $ 10^{-10}$  &   $41$    &   $4.5256e-10 $  &    $1.5511e-08 $\\ 
		$ 10^{-12} $ & $ 10^{-12}$  &   $55$    &   $4.4334e-12 $  &    $2.8596e-10 $\\ \hline
	\end{tabular}
	\caption{Example 1 -- error between true and incremental SVD}
	\label{table:Error_SVD ex1}
\end{table}
\begin{table}
	\renewcommand{\arraystretch}{1.25}
	\centering
	\begin{tabular}{llccc}
		\hline
		$  \mathrm{tol} $   &   $ \mathrm{tol}_\mathrm{sv}$  &   Rank    &   Exact error &   Incr.\ error bound   \\ \hline
		$ 10^{-8}$ &   $ 10^{-8}$  &    $35$   &  $3.0859e-07$ &    $3.6931e-06 $\\ 
		$ 10^{-8}$ &   $10^{-10}$  &    $66$   &  $1.3881e-07$ &    $1.1429e-06 $\\ 
		$ 10^{-8}$ &   $10^{-12}$  &    $64$   &  $3.4657e-07$ &    $1.5321e-06 $\\ 
		$ 10^{-10}$&   $10^{-8} $  &    $31$   &  $4.1497e-08$ &    $1.7368e-06 $\\ 
		$ 10^{-10}$&   $10^{-10}$  &    $45$   &  $5.3142e-10$ &    $3.6491e-08 $\\ 
		$ 10^{-10}$&   $ 10^{-12}$ &    $74$   &  $7.7348e-10$ &    $1.1523e-08 $\\ 
		$ 10^{-12}$&   $ 10^{-8} $ &    $30$     &  $4.1497e-08$ &    $1.7368e-06 $\\ 
		$ 10^{-12}$&   $ 10^{-10}$ &    $41$   &  $4.6086e-10$ &    $1.8671e-08  $\\ 
		$ 10^{-12} $ & $ 10^{-12} $ &   $59$   &  $4.8658e-12 $&     $3.4880e-10 $\\ \hline
	\end{tabular}
	\caption{Example 2 -- error between true and incremental SVD}
	\label{table:Error_SVD_ex2}
\end{table}
\begin{table}
	\renewcommand{\arraystretch}{1.25}
	\centering
	\begin{tabular}{llccc}
		\hline
		$  \mathrm{tol} $   &   $ \mathrm{tol}_\mathrm{sv}$  &   Rank    &   Exact error &   Incr.\ error bound   \\ \hline
		$10^{-8}$  &   $ 10^{-8}$  &    $38$   &  $6.5705e-08$ &    $4.3271e-06 $   \\ 
		$ 10^{-8}$ &   $ 10^{-10}$ &    $72$   &  $6.8271e-07$ &    $1.1523e-06 $   \\ 
		$ 10^{-8}$ &   $ 10^{-12}$ &    $67$  & $ 3.6916e-07$ &    $2.3847e-06$   \\ 
		$ 10^{-10}$&   $ 10^{-8} $ &    $31$   & $ 4.7018e-08$ &    $2.2388e-06 $   \\ 
		$ 10^{-10}$&   $ 10^{-10}$ &    $49$   & $ 4.8302e-10$ &    $4.3655e-08 $   \\ 
		$ 10^{-10}$&   $ 10^{-12}$ &    $78$   & $2.4473e-08 $ &    $2.6825e-08 $   \\ 
		$ 10^{-12}$ &  $ 10^{-8} $ &    $31$   & $4.7018e-08 $ &    $ 2.2388e-06$   \\ 
		$ 10^{-12}$ &  $ 10^{-10}$ &    $41$   & $4.9660e-10 $ &    $2.5022e-08 $    \\ 
		$ 10^{-12}$ &  $ 10^{-12}$ &    $60$   & $6.3200e-12 $ &    $5.7438e-10 $    \\ \hline
	\end{tabular}
	\caption{Example 3 -- error between true and incremental SVD}
	\label{table:Error_SVD_ex3}
\end{table}




\Cref{fig:errorPOD_modes_ex3} shows the exact and incrementally computed POD singular values and also the weighted norm error between the exact and incrementally computed POD modes with $ \mathrm{tol} $ and $ \mathrm{tol}_\mathrm{sv} $ both equal to $10^{-12}$.  The errors for the POD modes corresponding to the largest singular values are extremely small (approximately $10^{-12}$).  The errors in the POD modes increase slowly as the corresponding singular values approach zero.  There are many accurate POD modes; the first $30$ modes are computed to an accuracy level of at least $ 10^{-5} $.  The POD singular value and mode errors behaved similarly for other cases.
\begin{figure}
	\centering
	\subfloat[POD singular values]{\includegraphics[width=.5\linewidth]{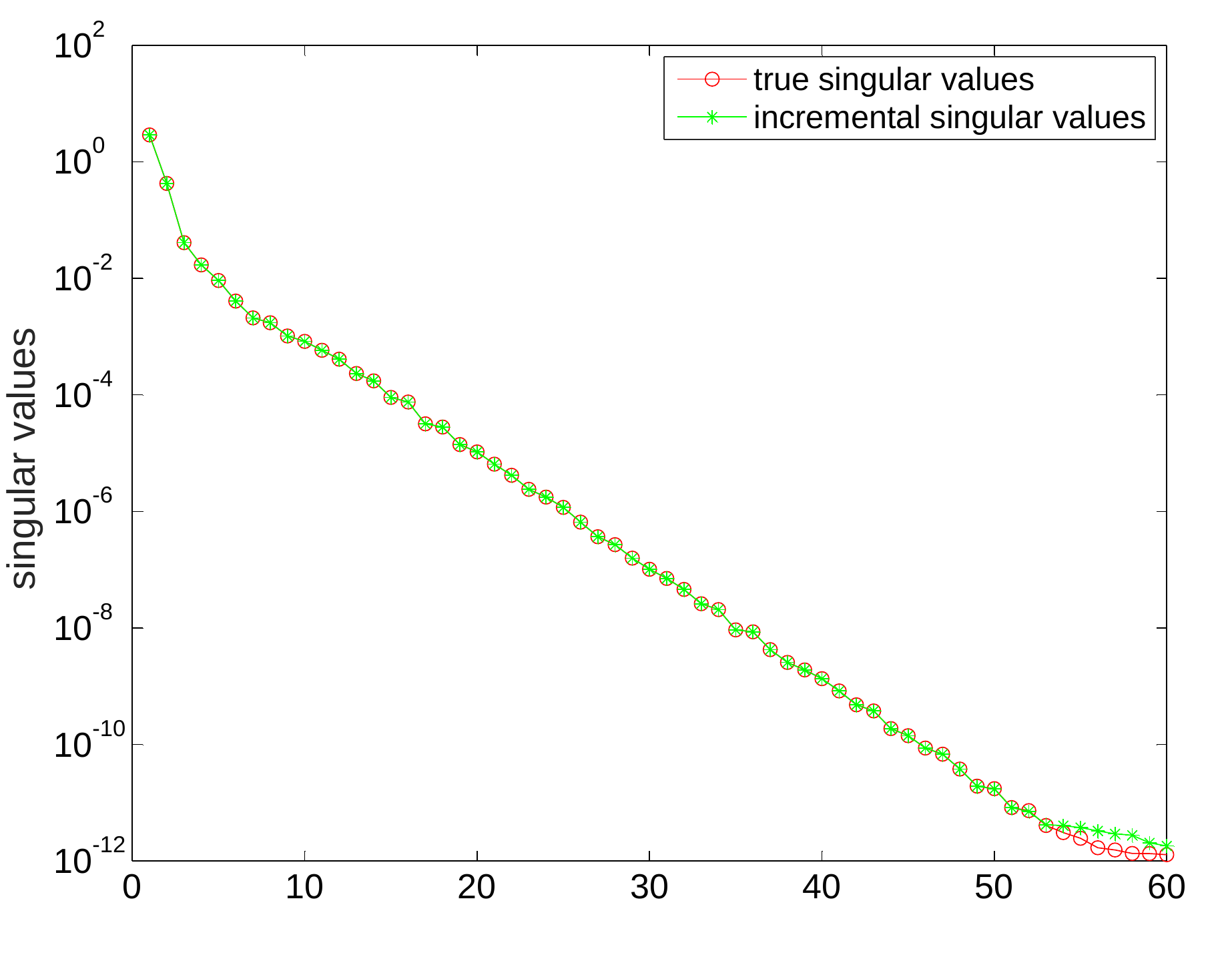}}
	\subfloat[POD mode errors]{\includegraphics[width=.5\linewidth]{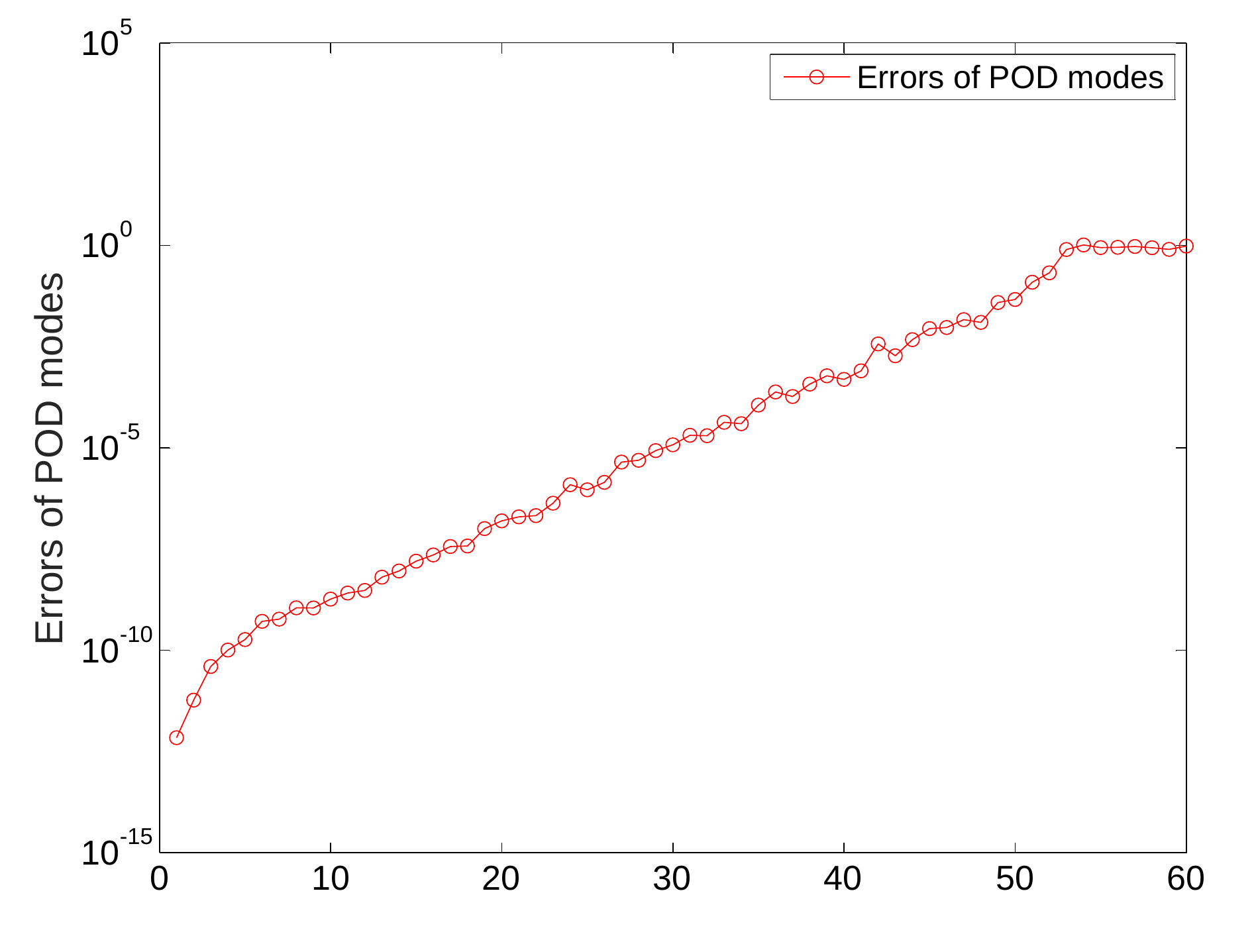}}
	\caption{\label{fig:errorPOD_modes_ex3}Example 3 -- exact versus incremental POD computations with $ \mathrm{tol} = \mathrm{tol}_\mathrm{sv} = 10^{-12}$}
\end{figure}
%

%
\section{Conclusion}
\label{sec:conclusion}
%
In our earlier work \cite{Fareed17}, we proposed computing the SVD with respect to a weighted inner product incrementally to obtain the POD eigenvalues and modes of a set of PDE simulation data.  In this work, we extended the algorithm to update the SVD and an error bound incrementally when a new column is added.  We also performed an error analysis of this algorithm by analyzing the error due to each individual truncation.  We showed that the algorithm produces the exact SVD of a matrix $ \tilde{U} $ such that $ \| U - \tilde{U} \|_{\mathcal{L}(\mathbb{R}^s,\mathbb{R}^m_M)} \leq e $, where $ U $ is the true data matrix, $ M $ is the weight matrix, and $ e $ is computed error bound.  We also proved error bounds for the incrementally computed singular values and singular vectors.  We tested our approach on three example data sets from a 1D FitzHugh-Nagumo PDE system with various choices of the two truncation tolerances.  In all of the tests, the incrementally computed error bound was larger than the exact error and the error bound was small.  Furthermore, the approximate singular values and dominant singular vectors were accurate.  Also, our analysis and the numerical tests suggest that there is no benefit from choosing one algorithm tolerance different than the other.


%
\section*{Acknowledgement}
  The authors thank Mark Opmeer for a helpful discussion.
%
%

\section{Appendix}
\label{Appen}
%
Let $ X $ and $ Y $ be two separable Hilbert spaces, with inner products $ (\cdot,\cdot)_X $ and $ (\cdot,\cdot)_Y $ and corresponding norms $ \| \cdot \|_X $ and $ \| \cdot \|_Y $.  Below, we drop the subscripts on the inner products and the norms since the space will be clear from the context.  Assume $ H, H_\varepsilon : X \to Y $ are compact linear operators.  In this section, we prove bounds on the error between the singular vectors of $ H $ and $ H_\varepsilon $ assuming the singular values are distinct.  Our results rely on techniques from \cite{glover88, Pedro08}.

Let $\{\sigma_k, v_k, w_k \}_{k \geq 1}$ and $\{\sigma^{\varepsilon}_k, v^{\varepsilon}_k, w^{\varepsilon}_k \}_{k \geq 1}$ be the ordered singular values and corresponding orthonormal singular vectors of $H$ and $H_\varepsilon$.  They satisfy
\begin{equation}\label{eqn:svd_H_Heps}
  H v_k = \sigma_k w_k,  \quad  H^* w_k = \sigma_k v_k,  \quad
  H_\varepsilon v^\varepsilon_k = \sigma^\varepsilon_k w^\varepsilon_k,  \qquad  H^*_\varepsilon w^\varepsilon_k = \sigma^\varepsilon_k v^\varepsilon_k,
\end{equation}
where the star denotes the Hilbert adjoint operator.  Also, if $ \sigma_k > 0 $, then $ \sigma_k^2 $ is the $ k $th ordered eigenvalue of the self-adjoint nonnegative compact operators $ H H ^* $ and $ H^* H $.  First, we recall a well-known bound on the singular values; see, e.g., \cite[page 30]{GohbergKrein69} and \cite[page 99]{GohbergGoldbergKaashoek90}.
\begin{proposition}
  Let $\varepsilon > 0 $ such that $ \| H - H_\varepsilon \|_{ \mathcal{L}( X, Y)} \leq \varepsilon $. Then for all $ k \geq 1 $ we have
  \begin{equation}\label{diff_kthsvalue}
  | \sigma_{k} - \sigma^{\varepsilon}_{k} | < \varepsilon.
  \end{equation}
\end{proposition}

In the results below, we require the singular vectors $ \{ v_k^\varepsilon, w_k^\varepsilon \} $ are suitably normalized.  We note that any pair $ \{ v_k^\varepsilon, w_k^\varepsilon \} $ of singular vectors for a fixed value of $ k $ can be rescaled by a constant of unit magnitude and remain a pair of singular vectors.  However, due to the relationship \eqref{eqn:svd_H_Heps}, we note that both vectors in the pair must be rescaled by the same constant.

The proof of the following result is largely contained in \cite[Appendix 2]{glover88}, but we include the proof here to be complete.
\begin{lemma}\label{lemma2}
%
Let $\varepsilon > 0 $ such that $ \| H - H_\varepsilon \|_{ \mathcal{L}( X, Y)} \leq \varepsilon $.  If $ \sigma_1 > \sigma_2 > 0 $, $ v_1^\varepsilon $ and $ w_1^\varepsilon $ are suitably normalized, and
\begin{equation}\label{eqn:eps_small_condition}
  \varepsilon \leq \frac{ \sigma_1 - \sigma_2 }{ 2 },
%
%
%
\end{equation}
then
\begin{equation}\label{diff_1stsvectors}
  \| v_1 - v^{\varepsilon}_1 \|  \leq  E_1^{1/2},  \quad  \| w_1 - w^{\varepsilon}_1 \|  \leq  E_1^{1/2} + 2 \sigma_1^{-1} \varepsilon,  \quad  E_1 = 2 \left( 1 - \sqrt{ \frac{ ( \sigma_1 - 2 \varepsilon )^2 - \sigma_2^2 }{ \sigma_1^2 - \sigma_2^2 } } \right).
%
%
%
%
%
\end{equation}
\end{lemma}
\begin{remark}
  	The larger error bound for $ \| w_1 - w_1^\varepsilon \| $ is due to the way we assume the singular vectors are normalized in the proof.  It is possible to use a different normalization and make the error bound larger for $ \| v_1 - v_1^\varepsilon \| $ instead.  We comment on the normalization in the proof.
\end{remark}
\begin{proof}
%
%
Define $ V_1 = \mathrm{span}\{ v_1 \} \subset X $.  We have $ X = V_1 \oplus V_1^{\perp} $, and therefore $ v^{\varepsilon}_1 = r_\varepsilon v_1 + x_\varepsilon $ for some constant $ r_\varepsilon $ and $ x_\varepsilon \in X $ satisfies $ ( x_\varepsilon , v_1 ) = 0 $.  This gives $ \| x_\varepsilon\|^2 = 1 - | r_\varepsilon |^2 $ and also $ | r_\varepsilon | \leq 1 $.  Then
\begin{align} \label{diff_v1}
\| v_1 - v^{\varepsilon}_1\|^2 & = \| v_1 - r_\varepsilon v_1 - x_\varepsilon \|^2 \nonumber \\
                             & = | 1 - r_\varepsilon|^2 \| v_1 \|^2 + \| x_\varepsilon \|^2 \nonumber\\
                             & = 2 ( 1 - \mathrm{Re}(r_\varepsilon) ).
\end{align}
Note $ \|\sigma^\varepsilon_1 w^\varepsilon_1\| = \| H_\varepsilon v^\varepsilon_1 \| $ implies
\begin{align*}
 %
           \sigma^\varepsilon_1 & = \| H_\varepsilon v^\varepsilon_1 + H v^\varepsilon_1 - H v^\varepsilon_1  \|  \\
                                           & \leq \|H v^\varepsilon_1 \| + \| H - H_\varepsilon \| \| v^\varepsilon_1  \|\\
                                           & \leq \|H (r_\varepsilon v_1 + x_\varepsilon) \| + \varepsilon \\
                                           & = \|r_\varepsilon \sigma_1 w_1 + H x_\varepsilon\| + \varepsilon.
\end{align*}
To estimate this norm, we use $ ( H x_\varepsilon , w_1) = ( x_\varepsilon , H^* w_1) = \sigma_1 ( x_\varepsilon , v_1) = 0 $ and also
$$
  \| H x_\varepsilon \|^2  =   \frac{ (H^* H x_\varepsilon, x_\varepsilon) }{ \| x_\varepsilon \|^2 }  \| x_\varepsilon \|^2   \leq  \sup_{ x \in V_1^\perp, \: x \neq 0 }   \frac{ (H^* H x, x) }{ \| x \|^2 }  \| x_\varepsilon \|^2 = \sigma_2^2 \, \| x_\varepsilon \|^2,
$$
where we used the variational characterization of the second eigenvalue $ \sigma_2^2 $ of the self-adjoint compact nonnegative operator $ H^* H $ \cite[Chapter 28]{Lax02}.  These results give
\begin{align*}
 %
  \|r_\varepsilon \sigma_1 w_1 + H x_\varepsilon \|^2 & = |r_\varepsilon|^2 \sigma^2_1 + \|H x_\varepsilon\|^2\\
                                                       & \leq  |r_\varepsilon|^2 \sigma^2_1 + \sigma_2^2 \, \|x_\varepsilon\|^2\\
                                                       & = \big( \sigma_1^2 - \sigma_2^2 \big) |r_\varepsilon|^2 + \sigma_2^2.
\end{align*}
Next, the assumption \eqref{eqn:eps_small_condition} for $ \varepsilon $ gives $ \varepsilon \leq (\sigma_1 - \sigma_2) / 2  \leq \sigma_1/2 $, and therefore $ \sigma_1 - 2 \varepsilon \geq 0 $.  Also, \eqref{diff_kthsvalue} gives $ -\varepsilon \leq \sigma_1^\varepsilon - \sigma_1 $, or $ \sigma_1^\varepsilon - \varepsilon \geq \sigma_1 - 2 \varepsilon \geq 0 $.  This gives $ ( \sigma_1^\varepsilon - \varepsilon )^2 \geq ( \sigma_1 - 2 \varepsilon )^2 $, and therefore
$$
  | r_\varepsilon |^2  \geq  \frac{ ( \sigma_1^\varepsilon - \varepsilon )^2 - \sigma_2^2 }{ \sigma_1^2 - \sigma_2^2 }  \geq  \frac{ ( \sigma_1 - 2 \varepsilon )^2 - \sigma_2^2 }{ \sigma_1^2 - \sigma_2^2 }.
$$
%
Note that the assumption \eqref{eqn:eps_small_condition} for $ \varepsilon $ guarantees that we can take a square root of this estimate.

If $ v_1^\varepsilon $ is normalized so that $ r_\varepsilon $ is a nonnegative real number, then \eqref{diff_v1}, $ 1-\mathrm{Re}(r_\varepsilon) = 1-| r_\varepsilon | $, and the above inequality give the desired estimate \eqref{diff_1stsvectors} for $ \| v_1 - v^{\varepsilon}_1\| $.  If $ r_\varepsilon $ is not a nonnegative real number, then rescale the singular vector pair $ \{v_1^\varepsilon,w_1^\varepsilon\} $ by $ \overline{r}_\varepsilon/ | r_\varepsilon | $ to obtain the proper normalization and the bound \eqref{diff_1stsvectors} for $ \| v_1 - v^{\varepsilon}_1\| $.

For $ w_1 $ and $ w_1^\varepsilon $, it does not appear that we can use a similar proof strategy since we have already rescaled the singular vector pair $ \{v_1^\varepsilon,w_1^\varepsilon\} $.  Specifically, we can obtain $ w_1^\varepsilon = s_\varepsilon w_1 + y_\varepsilon $, but it is not clear that $ s_\varepsilon $ will be a nonnegative real number and we are unable to rescale again.  Therefore, we use $ \| H \| = \sigma_1 $, $ \| H - H_\varepsilon \| \leq \varepsilon $, and $ | \sigma_1 - \sigma_1^\varepsilon | \leq \varepsilon $ to directly estimate:
\begin{align*}
  \| w_1 - w^{\varepsilon}_1 \|  &=  \| \sigma_1^{-1} H v_1 - (\sigma^\varepsilon_1)^{-1} H_\varepsilon v^{\varepsilon}_1 \|\\
  	&\leq  \|  \sigma_1^{-1} H v_1 - \sigma_1^{-1} H v^{\varepsilon}_1 \| + \| \sigma_1^{-1} H v^{\varepsilon}_1 - \sigma_1^{-1} H_\varepsilon v^{\varepsilon}_1 \| + \| \sigma_1^{-1} H_\varepsilon v^{\varepsilon}_1 - (\sigma^\varepsilon_1)^{-1} H_\varepsilon v^{\varepsilon}_1 \|\\
    &\leq  \| v_1 - v_1^\varepsilon \|  +  \sigma_1^{-1} \varepsilon  +  | \sigma_1^\varepsilon \sigma_1^{-1} - 1 |\\
    &\leq  \| v_1 - v_1^\varepsilon \|  +  2 \sigma_1^{-1} \varepsilon.
\end{align*}
\end{proof}

In the result below, note that $ \varepsilon_1 = \varepsilon $ and $ E_1 $ is defined as in \eqref{diff_1stsvectors} in \Cref{lemma2} above.
\begin{theorem}
	Let $ k \geq 1 $, and let $ \varepsilon > 0 $ such that $ \| H - H_\varepsilon \|_{ \mathcal{L}( X, Y)} \leq \varepsilon $.  For $ j = 1, \ldots, k $, define
	$$
	  \varepsilon_j = j \varepsilon + 2 \sum_{i=1}^{j-1} \left( \varepsilon_i + \sigma_{i} E_i^{1/2} \right),  \quad  E_j = 2 \left( 1 - \sqrt{ \frac{ ( \sigma_j - 2 \varepsilon_j )^2 - \sigma_{j+1}^2 }{ \sigma_j^2 - \sigma_{j+1}^2 } } \right).
	$$
	If the first $ k+1 $ singular values of $ H $ are distinct and positive, the singular vector pairs $ \{ v_j^\varepsilon, w_j^\varepsilon \}_{j=1}^k $ are suitably normalized, and
	$$
	  \varepsilon_j \leq \frac{ \sigma_j - \sigma_{j+1} }{ 2 }  \quad  \mbox{for $ j = 1, \ldots, k $,}
	$$
	then
	\begin{equation}\label{diff_kthsvectors}
	  \| v_j - v^{\varepsilon}_j \|  \leq  E_j^{1/2},  \quad  \| w_j - w^{\varepsilon}_j \|  \leq  E_j^{1/2} + 2 \sigma_j^{-1} \varepsilon_j,  \quad  \mbox{for $ j = 1, \ldots, k $.}	
	\end{equation}
\end{theorem}
\begin{proof}
	The proof is by induction.  First, the result is true for $ k = 1 $ by \Cref{lemma2}.  Next, assume the result is true for all $ j = 1 , \ldots , k-1 $.  Define compact linear operators for $ j = 2, \ldots, k $ by
	$$
	  H^{j}x = H x - \sum_{i=1}^{j-1} \sigma_{i}(x,v_i) w_i,  \quad  H^{j}_\varepsilon x = H_\varepsilon x - \sum_{i=1}^{j-1} \sigma^{\varepsilon}_{i}(x,v^{\varepsilon}_i) w^{\varepsilon}_i,
	$$
	for all $ x \in X $.  Then the ordered singular values and corresponding singular vectors of $ H^j $ and $ H^j_\varepsilon $ are $ \{ \sigma_i, v_i, w_i \}_{i \geq j} $ and $ \{ \sigma_i^\varepsilon, v_i^\varepsilon, w_i^\varepsilon \}_{i \geq j} $.
	
	Note that
	\begin{align*}
	  \| H^{k} x -  H^{k}_\varepsilon x \|  &\leq  \| ( H - H_\varepsilon ) x    \| +  \sum_{i=1}^{k-1} \| \sigma^{\varepsilon}_{i}(x,v^{\varepsilon}_i) w^{\varepsilon}_i  -  \sigma_{i}(x,v_i) w_i  \| \\
	  &\leq \varepsilon \| x \| + \| x \|  \sum_{i=1}^{k-1} \big( | \sigma^{\varepsilon}_{i} - \sigma_{i} | + \sigma_{i} \|v^{\varepsilon}_i - v_i \| + \sigma_{i} \|w^{\varepsilon}_i - w_i \| \big).
	\end{align*}
	Then since the result \eqref{diff_kthsvectors} is true for all $ j = 1 , \ldots , k-1 $, we have $ \| H^k - H^k_\varepsilon \| \leq \varepsilon_k $, where
	\begin{align*}
	  \varepsilon_k  &=  \varepsilon + \sum_{i=1}^{k-1} \left( \varepsilon + \sigma_{i} E_i^{1/2} + \sigma_{i} \big( E_i^{1/2} + 2 \sigma_i^{-1} \varepsilon_i \big) \right) \\
	  &= k \varepsilon + 2 \sum_{i=1}^{k-1} \left( \varepsilon_i + \sigma_{i} E_i^{1/2} \right).
	\end{align*}
	Applying \Cref{lemma2} to $ H^k $ and $ H^k_\varepsilon $ with $ \| H^k - H^k_\varepsilon \| \leq \varepsilon_k $ completes the proof.
\end{proof}

\bibliographystyle{plain}
\bibliography{incremental_POD}

\end{document}